\documentclass[10pt,leqno]{article}

\usepackage{amsmath,amssymb,amsthm,mathrsfs,dsfont}

\usepackage[margin=3cm]{geometry} 

\usepackage{titlesec,hyperref}

\usepackage{color}

\usepackage{fancyhdr}
\usepackage{subfigure}
\usepackage[graphicx]{realboxes}
\usepackage{float}

\titleformat{\subsection}{\it}{\thesubsection.\enspace}{1pt}{}

\newtheorem{theo}{Theorem}[section]
\newtheorem{lemm}[theo]{Lemma}
\newtheorem{defi}[theo]{Definition}
\newtheorem{coro}[theo]{Corollary}
\newtheorem{prop}[theo]{Proposition}

\numberwithin{equation}{section}

\allowdisplaybreaks 

\begin{document}
\title{The local well-posedness, blow-up and global solution of a new integrable system in Besov spaces
	\hspace{-4mm}
}

\author{
	Pei $\mbox{Zheng}^1$ \footnote{Email: zhengp25@mail2.sysu.edu.cn},\quad
	Zhaoyang $\mbox{Yin}^{1,2}$\footnote{E-mail: mcsyzy@mail.sysu.edu.cn}
	\\
$^1\mbox{Department}$ of Mathematics,
Sun Yat-sen University, Guangzhou 510275, China\\
$^2\mbox{School}$ of Science,\\ Shenzhen Campus of Sun Yat-sen University, Shenzhen 518107, China}

\date{}
\maketitle
\hrule

\begin{abstract}
	In this paper, we first establish the local well-posednesss for the Cauchy problem of a $N$-peakon system in the sense of Hadamard in both critical Besov spaces and supercritical Besov spaces. Second, we gain a blow-up criterion. According to blow-up criterion wemreach a precise blow-up criterion. Under a sign condition, we reach the existence of global solution. Finally, based on the first-order difference method, we give a simulation example of the blow-up of the equation and the properties of the global solution.
%
%
\end{abstract}

\vspace*{10pt}

\tableofcontents

	\section{Introduction}
	Korteweg-de Vries (KdV) equation which was introduced to describe the behavior of long waves on shallow water in the most famous model in soliton theory, because it is integrable and includes the phenomena of soliton interaction. But the KdV equation cannot model the occurrence of breaking waves. In 1993, Camassa and Holm obtain a nonlinear partial differential equation \cite{CH}
	
	\begin{equation}
		m_{t}+um_{x}+2u_{x}m=0, m=u-u_{xx}+\kappa\nonumber
	\end{equation}
	which is named Camassa-Holm (CH) equation.
	
	A particular feature of the CH equation is that when $\kappa=0$ it admits peaked soliton solutions which are also called peakons. It can be regarded as a shallow water wave equation.\cite{CH,CH2,CH3}. Its complete integrability was discussed in \cite{CH,CH4,CH5}. The CH equation also has a Hamilton structure\cite{CH6,CH7}, and admits exact orbitally stable peaked solitons of the form $ce^{|x-ct|}$ with $c>0$. The local well-posedness of the Cauchy problem of CH equation in Besov spaces and Sobolev spaces was proved in \cite{CH8,CH9,CH10,CH11}. 
	
	Since constructing peakon equation are interested in the fields of Physics and Mathematics. Based on an asymptotic integrability approach, Degasperis and Procesi discovered a new equation which has peakon solutions\cite{DP1,DP2}. Similarly, the Cauchy problem of the DP equation is locally well-posed in certain Sobolev and Besov spaces\cite{DP3,DP4,DP5}. And Geng and Xue constructed a two component peakon with cubic nonlinearity\cite{cubic}, three components generalization of CH equation\cite{three} and super CH equation with $N$-peakon\cite{superN}.
	
	Geng and Xue constructed from the compatibility of a $2\times 2$ matrix spectral problem, a new nonlinear evolution with $N$-peakon is derived\cite{Npeakon}
	\begin{equation}\label{Npeakon}
		\left\{
		\begin{array}{l}
			n_{t}=4[n(v_{x}+2\beta_{0}v)]_{x}\\
			n=4\beta_{0}^{2}v-v_{xx}\\
			v|_{t=0}=v_{0}
		\end{array}
			\right.
	\end{equation}
	It is shown in that system \eqref{Npeakon} admits exact solutions with $N$-peakon, which takes the form 
	$$v(x,t)=\sum_{i=1}^Nn_i(t)e^{-2\beta_0|x-x_i(t)|}$$
	where $n_i$ and $x_i$ evolve according to a dynamical system.
	
	After some transformations, (1.1) can be changed into a transport-like equation
	\begin{equation}\label{Npeakon2}
		\left\{
		\begin{array}{l}
			v_{t}-(8\beta_{0}v+2v_{x})v_{x}=-8\beta_{0}^{2}v^{2}+8\beta_{0}P_{1}(D)(2\beta_{0}^{2}v^{2}+v_{x}^{2})+8\beta^{2}_{0}P_{2}(D)(4\beta_{0}^{2}v^{2}-v_{x}^{2})\\
			v|_{t=0}=v_{0}
		\end{array}
		\right.
	\end{equation}
	with $P_{2}(D)=(4\beta_{0}^{2}-\partial_{x}^{2})^{-1}$, $P_{1}(D)=\partial_{x}P_2(D)$.
	
	In this paper, firstly, applying Littlewood-Paley theory and transport theory, for the initial data in certain Besov spaces of high regularity or of critical regularity, complete the locally well-posedness in Besov space $B^{s}_{p,r}$ with $\{s\geq\frac 1 2, p>2, 1\leq r\leq\infty\}$ or $\{s>\frac 1 p, 1\leq p\leq2, 1\leq r\leq\infty\}$ or $\{s=\frac 1 p, 1\leq p\leq2, r=1\}$, we prove the local solution to equation \eqref{Npeakon} exists uniquely and depends continuously on the initial data. And we will talk about the blow-up criterion and the blow up condition of initial data. In the last part of the paper,according to the structure of equation, we can find a sign-preserved property, and by the virtue of the high-preserved we see that the $H^{1}$-norm of v is non-increasing, then we can finally give the condition of existence of the global solution.

	\section{Preliminaries}
	
	In this section, we will present some propositions about the Littlewood-Paley decomposition and the non homogeneous Besov spaces with their properties.
	\begin{prop}[Littlewood-Paley decomposition]\cite{BCD,He}
		There exists a couple of smooth function $(\chi,\varphi)$ valued in $[0,1]$, such that $\kappa$ is supported in the ball $B\triangleq\left\{\xi\in\mathbb{R}^d: \frac 3 4\leq|\xi|\leq\frac 8 3\right\}$. Moreover\\
		$$
			\forall\xi\in\mathbb{R}^d, \chi(\xi)+\sum_{j\geq0}\varphi(2^{-j}\xi)=1$$
		and
		$$
			supp\ \varphi(2^{-j}\cdot)\cap supp\ \varphi(2^{-j^{\prime}\cdot})=\emptyset,\ if\  |j-j^{\prime}|\geq2$$
		$$
			supp\ \chi(\cdot)\cap supp\ \varphi(2^{-j^{\prime}\cdot})=\emptyset,\ if\ j\geq1$$
	\end{prop}
		
		~~Then for all $u\in\mathcal{S}^\prime$, we can define the nonhomogeneous dyadic blocks as follows. Let
		$$
			\Delta_j u\triangleq0,\ if\ j\leq-2,\ \Delta_{-1}u\triangleq\chi(D)u=\mathcal{F}^{-1}(\chi\mathcal{F}u)$$
		$$
			\Delta_{j}\triangleq\varphi(2^{-j}u)=\mathcal{F}^{-1}(\varphi(2^{-j}\cdot)\mathcal{F}u),\ if\ j\geq 0$$
		Hence,
		$$
			u=\sum_{j\in\mathbb{Z}}\Delta_{j}u\ in\ \mathcal{S}^\prime(\mathbb{R}^d)$$
		where the right-hand side is called the nonhomogeneous Littlewood-Paley decomposition of $u$.
		
		~~According to Young's inequality, we get
		$$
			\Vert\Delta_j u\Vert_{L^p},\Vert S_j u\Vert_{L^p}\leq C\Vert u\Vert_{L^p},\ \forall 1\leq p\leq\infty$$
		The low frequency cut-off operator $\mathcal{S}_j$ is defined by
		$$
			S_j u\triangleq \sum_{j^\prime=-1}^{j-1}\Delta_{j^\prime}u=\chi(2^{-j}D)u=\mathcal{F}^{-1}(\chi(2^{-j}\xi)\mathcal{F}u),\ \forall j\in\mathbb{N}$$
		and the Littlewood-Paley decomposition is quasi-orthogonal in $L^2$ in the following sense:
		$$
			\Delta_{j}\Delta_{k}u\equiv0,\ if\ |j-k|\geq 2
		$$
		$$
			\Delta_{j}(S_{k-1}u\Delta_{k}u)\equiv0,\ if\ |j-k|\geq 5$$
		
		\begin{defi}[Nonhomogeneous Besov space]\cite{BCD}
			Let $s\in\mathbb{R}$, $1\leq p,\ r\leq\infty$. The nonhomogeneous Besov space $B^s_{p,r}(\mathbb{R}^d)${\rm(}$B^s_{p,r}$ for short{\rm)} is defined by
			$$
				B^s_{p,r}(\mathbb{R}^d)\triangleq\left\{f\in\mathcal{S}^\prime:\Vert f\Vert_{B^s_{p,r}}=\triangleq\Vert(2^{js}\Vert\Delta_j f\Vert_{L^p})_{j\geq -1}\Vert_{l^r}<\infty\right\}$$
			If $s=\infty$, we have
			 $$B^\infty_{p,r}\triangleq\cap_{s\in\mathbb{R}}B^s_{p,r}=\left\{f\in\mathcal{S}^\prime:\sup_{j\geq -1}2^{js}\Vert\Delta_j f\Vert_{L^p}<\infty\right\}$$
		\end{defi}
		
		In the following lemma, we list some important properties of Besov spaces.
		
		\begin{lemm}\label{Besov}\cite{BCD,He}
			Let $s\in\mathbb{R},\ 1\leq p,p_1,p_2,r,r_1,r_2\leq\infty$. We have  \\
			{\rm(1)} $B^s_{p,r}$ is a Banach space, and is continuously embedded in $\mathcal{S}'$. \\
			{\rm(2)} If $r<\infty$, then $\lim\limits_{j\rightarrow\infty}\|S_j u-u\|_{B^s_{p,r}}=0$. If $p,r<\infty$, $C_0^{\infty}$ is dense in $B^s_{p,r}$. \\
			{\rm(3)} If $p_1\leq p_2$ and $r_1\leq r_2$, then $ B^s_{p_1,r_1}\hookrightarrow B^{s-d(\frac 1 {p_1}-\frac 1 {p_2})}_{p_2,r_2}$. 
			If $s_1<s_2$, the embedding $B^{s_2}_{p,r_2}\hookrightarrow B^{s_1}_{p,r_1}$ is locally compact. \\
			{\rm(4)} $\forall s>0$, $B^s_{p,r}\cap L^\infty$ is an algebra. Moreover, $B^s_{p,r}\hookrightarrow L^{\infty} \Leftrightarrow s>\frac d p\ \text{or}\ s=\frac d p,\ r=1$ $\Leftrightarrow B^s_{p,r}$ is an algebra. \\
			{\rm(5)} Complex interpolation: $\forall f\in B^{s_1}_{p,r}\cap B^{s_2}_{p,r},\ \forall\theta\in[0,1]$
			\begin{equation}
				\Vert f\Vert_{B^{\theta s_1+(1-\theta) s_2}_{p,r}}\leq\Vert f\Vert^\theta_{B^{s_1}_{p,r}}\Vert f\Vert^{1-\theta}_{B^{s_2}_{p,r}}
			\end{equation}\\
			{\rm(6)} Logarithm interppolation: $\forall s\in\mathbb{R},\ \varepsilon>0$, and $1\leq p\leq\infty$, there exists a constant $C$ such that 
			\begin{equation}
				\Vert u\Vert_{B^s_{p,1}}\leq C\frac{1+\varepsilon} {\varepsilon}\Vert u\Vert_{B^s_{p,\infty}}\left(1+\log\frac{\Vert u\Vert_{B^{s+\varepsilon}_{p,\infty}}}{\Vert u\Vert_{B^{s}_{p,\infty}}}\right)
			\end{equation}\\
			{\rm(7)} Fatou property: if $(u_n)_{n\in\mathbb{N}}$ is a bounded sequence in $B^s_{p,r}$, then an element $u\in B^s_{p,r}$ and a subsequence $(u_{n_k})_{k\in\mathbb{N}}$ exists such that
			$$ \lim_{k\rightarrow\infty}u_{n_k}=u\ \text{in}\ \mathcal{S}'\quad \text{and}\quad \|u\|_{B^s_{p,r}}\leq C\liminf_{k\rightarrow\infty}\|u_{n_k}\|_{B^s_{p,r}}. $$
			{\rm(8)} Let $m\in\mathbb{R}$ and $f$ be a $S^m$- multiplier, (i.e. $f$ is a smooth function and satisfies that $\forall\alpha\in\mathbb{N}^d$, $\exists C=C(\alpha)$, such that $|\partial^{\alpha}f(\xi)|\leq C(1+|\xi|)^{m-|\alpha|},\ \forall\xi\in\mathbb{R}^d)$.
			Then the operator $f(D)=\mathcal{F}^{-1}(f\mathcal{F}\cdot)$ is continuous from $B^s_{p,r}$ to $B^{s-m}_{p,r}$.
		\end{lemm}
		
		\begin{prop}\label{multiplier}
			$\forall \beta_0\neq0$, we have $(4\beta_0^2-\partial_x^2)^{-1}$ is a $S^{-2}$-multiplier, and $\partial_{x}(4\beta_0^2-\partial_x^2)^{-1}$ is a $S^{-1}$-multiplier.
		\end{prop}
		
	Next, we introduce the paradifferential calculus of Besov space and the main continuity properties of the paraproduct and the remainder.
	
	\begin{defi}\cite{BCD}
		The nonhomogeneous paraproduct of v by u is defined by
		$$T_u v= \sum_j S_{j-1}u\Delta_j v$$
		The nonhomogeneous remainder of u and v is defined by
		$$R(u,v)=\sum_{|k-j|\leq1}\Delta_k u\Delta_j v$$
	\end{defi}
	
	\begin{lemm}\label{Bony}\cite{BCD}
		{\rm (1)} $\forall t<0,s\in\mathbb{R},u\in B^t_{p,r_1}\cap L^\infty,v\in B^s_{p,r_2}$ with $\frac 1 r=\frac 1 {r_1}+\frac 1 {r_2}$, then
		$$\|T_u v\|_{B^s_{p,r_2}}\leq C\|u\|_{L^\infty}\|v\|_{B^s_{p,r_2}}$$
		or
		$$\|T_u v\|_{B^{s+t}_{p,r}}\leq C\|u\|_{B^t_{\infty,r_1}}\|v\|_{B^s_{p,r_2}}$$
		{\rm (2)} $\forall s_1,s_2\in\mathbb{R},1\leq p_1,p_2,r_1,r_2\leq\infty$, with $\frac 1 p=\frac 1 {p_1}+\frac 1 {p_2}\leq1,\frac 1 r=\frac 1 {r_1}+\frac 1 {r_2}\leq1$. Then $\forall (u,v)\in B^{s_1}_{p_1,r_1}\times B^{s_2}_{p_2,r_2}$, if $s_1+s_2>0$
		$$\|R(u,v)\|_{B^{s_1+s_2}_{p,r}}\leq C\|u\|_{B^{s_1}_{p_1,r_1}}\|v\|_{B^{s_2}_{p_2,r_2}}$$
		If $r=1$ and $s_1+s_2=0$,
		$$\|R(u,v)\|_{B^0_{p,\infty}}\leq C\|u\|_{B^{s_1}_{p_1,r_1}}\|v\|_{B^{s_2}_{p_2,r_2}}$$
	\end{lemm}
	
	We then have the following product laws:
	\begin{lemm}\label{product}\cite{BCD,He}
		{\rm(1)} For any $s>0$ and any $(p,r)$ in $[1,\infty]^2$, the space $L^{\infty} \cap B^s_{p,r}$ is an algebra, and a constant $C$ exists such that
		$$ \|uv\|_{B^s_{p,r}}\leq C(\|u\|_{L^{\infty}}\|v\|_{B^s_{p,r}}+\|u\|_{B^s_{p,r}}\|v\|_{L^{\infty}}). $$
		{\rm(2)} If $1\leq p,r\leq \infty,\ s_1\leq s_2,\ s_2>\frac{d}{p} (s_2 \geq \frac{d}{p}\ \text{if}\ r=1)$ and $s_1+s_2>\max(0, \frac{2d}{p}-d)$, there exists $C$ such that
		$$ \|uv\|_{B^{s_1}_{p,r}}\leq C\|u\|_{B^{s_1}_{p,r}}\|v\|_{B^{s_2}_{p,r}}. $$
		{\rm(3)} If $1\leq p\leq 2$, $\forall u\in B^{\frac d p -d}_{p,\infty}(\mathbb{R}^d)$, $\forall v\in B^{\frac d p}_{p,1}(\mathbb{R}^d)$, there exists $C$ such that
		$$ \|uv\|_{B^{\frac d p-d}_{p,\infty}}\leq C \|u\|_{B^{\frac d p-d}_{p,\infty}}\|v\|_{B^{\frac d p}_{p,1}}. $$		
	\end{lemm}
	
	\begin{prop}\label{dual}\cite{BCD}
		Let $s\in\mathbb{R},\ 1\leq p,r\leq\infty.$
		\begin{equation}\left\{
			\begin{array}{ll}
				B^s_{p,r}\times B^{-s}_{p',r'}&\longrightarrow\mathbb{R},  \\
				(u,\phi)&\longmapsto \sum\limits_{|j-j'|\leq 1}\langle \Delta_j u,\Delta_{j'}\phi\rangle,
			\end{array}\right.
		\end{equation}
		defines a continuous bilinear functional on $B^s_{p,r}\times B^{-s}_{p',r'}$. Denoted by $Q^{-s}_{p',r'}$ the set of functions $\phi$ in $\mathcal{S}'$ such that
		$\|\phi\|_{B^{-s}_{p',r'}}\leq 1$. If $u$ is in $\mathcal{S}'$, then we have
		$$\|u\|_{B^s_{p,r}}\leq C\sup_{\phi\in Q^{-s}_{p',r'}}\langle u,\phi\rangle.$$
	\end{prop}
	
	We now state the definition of Osgood modulus of continuity and Osgood lemma.
	
	\begin{defi}\cite{BCD}
		Let $a\in (0,1]$. Nondecreasing nonzero continuous function $\mu$ is defined on $[0,a]$ and $\mu:[0,a]\rightarrow\mathbb{R}^+\ s.t.\ \mu(0)=0$, then $\mu$ is called a modulus of continuity. 
		And if $\mu$ is called an Osgood modulus of continuity if
		$$\int_0^a \frac{\rm{d}r}{\mu(r)}=\infty$$
	\end{defi}
	
	The Osgood lemma are useful to prove the uniqueness of some ordinary differential equations and one can utilize a procedure to construct a large families of Osgood functions.
	
	\begin{lemm}[Osgood lemma]\label{Osgood}\cite{BCD}
		Let $\rho$ be a measurable function mapping $[t_0,T]$ to $[0,a]$, $\gamma$ is a locally integrable function from $[t_0,T]$ to $\mathbb{R}^+$, and $\mu$ is a nondecreasing function mapping $[0,a]$ to $\mathbb{R}^+$. Assume that, for some nonnegative real number $c$, the function $\rho$ satisfies
		$$\rho(t)\leq c+\int_{t_0}^t\gamma(t^\prime)\mu(\rho(t^\prime)){\rm d}t^\prime\quad for\ a.e.\ t\in[t_0,T]$$
		$\bullet$ If $c$ is positive 
		$$-\mathcal{M}(\rho(t))+\mathcal{M}(c)\leq\int_{t_0}^{t}\gamma(t^\prime){\rm d}t^\prime\quad with\quad\mathcal{M}(x)=\int_x^a\frac{{\rm d}r}{\mu(r)}\quad for\ a.e.\ t\in[t_0,T]$$
		$\bullet$ If $c=0$ and $\mu$ is an Osgood modulus of continuity, then $\rho=0,\ a.e.\ t\in[t_0,T]$.
	\end{lemm}
	
	Next, we state some useful result in the transport equation theory, which are crucial to the proofs of our main thorems later. Let us consider the linear transport equation:
	\begin{equation}
		\left\{
		\begin{array}{l}\label{transport}
			\partial_t f+v\cdot \nabla f=g\\
			f|_{t=0}=f_0
		\end{array}
		\right.
	\end{equation}
	
	\begin{lemm}[A priori estimates in Besov spaces]\label{priori estimate}\cite{BCD}
		Let $s\in\mathbb{R},\ 1\leq p,r\leq\infty$. Assume that
		$$ s\geq -d\min(\frac{1}{p_1},\frac{1}{p^\prime})$$
		
		There exists a constant $C$ such that for all solutions $f\in L^{\infty}([0,T];B^s_{p,r})$ of \eqref{transport} in $d$-dimensional space with initial data $f_0\in B^s_{p,r}$, and $g\in L^1([0,T];B^s_{p,r})$, we have for a.e. $t\in[0,T]$,
		$$ \|f(t)\|_{B^s_{p,r}}\leq \|f_0\|_{B^s_{p,r}}+\int_0^t \|g(t')\|_{B^s_{p,r}}\rm{d}t'+\int_0^t V_{p_1}^{'} (t^{'})\|f(t)\|_{B^s_{p,r}}{\rm d}t{'} $$
		or
		$$ \|f(t)\|_{B^s_{p,r}}\leq e^{CV_{p_1}(t)}\Big(\|f_0\|_{B^s_{p,r}}+\int_0^t e^{-CV_{p_1}(t')}\|g(t')\|_{B^s_{p,r}}{\rm d}t'\Big) $$
		with
		\begin{equation*}
			V_{p_1}'(t)=\left\{\begin{array}{ll}
				
				\|\nabla v\|_{B^{\frac{d}{p_1}}_{p_1,\infty}\cap L^\infty},\ &\text{if}\ s<1+\frac{d}{p_1}\\
				\|\nabla v\|_{B^{s-1}_{p_1,r}},\ &\text{if}\ s>1+\frac{d}{p_1}\ {\rm or}\  \left\{ s=1+\frac{d}{p_1} \ {\rm and}\  r=1\right\}
			\end{array}\right.
		\end{equation*}
		If $f=v,$ for all $s>0$
		$$ V_{p_1}^{'}(t)=\|\nabla v(t)\|_{L^{\infty}}.$$
	\end{lemm}
	
	\begin{lemm}\label{refine priori}\cite{LuoRefine,He}
		If $$\{s\geq\frac d 2, p>2, 1\leq r\leq\infty\}\ or\ \{s>\frac d p, 1\leq p\leq2, 1\leq r\leq\infty\}\ or\ \{s=\frac d p, 1\leq p\leq2, r=1\}$$
		then for the solution $f\in L^\infty([0,T];B^s_{p,r})$ of \eqref{transport} with $v\in L^1([0,T];B^{s+1}_{p,r})$, then the initial data $f_0\in B^s_{p,r})$ and $g\in L^1([0,T];B^s_{p,r}$, we then have
		$$ \| f(t)\|_{B^s_{p,r}}\leq e^{CV(t)}(\|f_0\|_{B^s_{p,r}}+\int_0^t e^{-CV(t)}\|g(\tau)\|_{B^s_{p,r}}{\rm d}\tau)$$
		with $V(t)=\int_0^t \|v(\tau)\|_{B^{s+1}_{p,r}}{\rm d}\tau$ and $C=C(p,r,s)$.
	\end{lemm}
	\begin{lemm}\label{existence}\cite{BCD}
		Let $1\leq p\leq p_1\leq\infty,\ 1\leq r\leq\infty,\ s> -d\min(\frac 1 {p_1}, \frac 1 {p'})$. Let $f_0\in B^s_{p,r}$, $g\in L^1([0,T];B^s_{p,r})$, and let $v$ be a time-dependent vector field such that $v\in L^\rho([0,T];B^{-M}_{\infty,\infty})$ for some $\rho>1$ and $M>0$, and
		$$
		\begin{array}{ll}
			\nabla v\in L^1([0,T];B^{\frac d {p_1}}_{p_1,\infty}), &\ \text{if}\ s<1+\frac d {p_1}, \\
			\nabla v\in L^1([0,T];B^{s-1}_{p,r}), &\ \text{if}\ s>1+\frac d {p_1}\ {\rm or}\ (s=1+\frac d {p_1}\ {\rm and}\ r=1).
		\end{array}
		$$
		Then the equation \eqref{transport} has a unique solution $f$ in \\
		$\bullet$ the space $C([0,T];B^s_{p,r})$, if $r<\infty$; \\
		$\bullet$ the space $\Big(\bigcap_{s'<s}C([0,T];B^{s'}_{p,\infty})\Big)\bigcap C_w([0,T];B^s_{p,\infty})$, if $r=\infty$.
	\end{lemm}
	
	\begin{lemm}\label{continuous}\cite{continuous}
		Let $1\leq p\leq\infty, 1\leq r\leq\infty,s>1+\frac d p {\rm (} or\ s=1+\frac d p,r=1,1\leq p<\infty{\rm )}$. Denote $\bar{\mathbb{N}}=\mathbb{N}\cup\{\infty\}$. Let $\{v^n\}_{n\in\bar{\mathbb{N}}}\subset C([0,T];B^{s-1}_{p,r})$. Assume that $v^n$ is the solution to
			$$
			\left\{
			\begin{array}{l}
				
				\partial_t v^n+a^n\partial_{x}v^n=f\\
				v^n|_{t=0}=v_0
				
			\end{array}
			\right.
			$$
		with $v_0\in B^{s-1}_{p,r},\ f\in L^1([0,T];B^{s-1}_{p,r})$ and for some $\alpha\in L^1([0,T])$
		$$\sup_{n\in\bar{\mathbb{N}}}\|a^n\|_{B^s_{p,r}}\leq \alpha(t)$$
		If $a^n\rightarrow a^\infty$ in $L^1([0,T];B^{s-1}_{p,r})$ as $n$ tends to $\infty$, then $v^n\rightarrow v^\infty$ in $C([0,T];B^{s-1}_{p,r})$ as $n$ tends to $\infty$.
	\end{lemm}

\section{Local well-posedness}
In this section, we will establish the local well-posedness for the Cauchy problem \eqref{Npeakon} or \eqref{Npeakon2} in Besov spaces. We want to solve the Cauchy problem \eqref{Npeakon2} directly by general Picard scheme and transport equation theory, but we will fail to do that due to the loss of regularity of the transport term. By Prop \ref{multiplier}, we can overcome the difficulty by solving the Cauchy problem \eqref{Npeakon}.

Firstly, we give a definition as follows
\begin{defi}
	Let $T>0,s\in\mathbb{R}$ and $1\leq p,r\leq\infty$. Denote
	$$E^s_{p,r}(T)\triangleq\left\{
	\begin{array}{ll}
		C([0,T];B^s_{p,r})\cap C^1([0,T];B^{s-1}_{p,r})\quad &if\ r<\infty\\
		C_w([0,T];B^s_{p,r})\cap C^{0,1}([0,T];B^{s-1}_{p,r})\quad &if\ r=\infty
	\end{array}\right.$$
	$$E^s_{p,r}(T_{-})\triangleq\left\{
	\begin{array}{ll}
		C([0,T);B^s_{p,r})\cap C^1([0,T);B^{s-1}_{p,r})\quad &if\ r<\infty\\
		C_w([0,T);B^s_{p,r})\cap C^{0,1}([0,T);B^{s-1}_{p,r})\quad &if\ r=\infty
	\end{array}\right.$$
\end{defi}
\begin{theo}\label{well posedness}
	Let $1\leq p,r\leq\infty,s\in\mathbb{R}$ and $(s,p,r)$ satisfies the condition
	\begin{equation}\label{index}
		\{s\geq\frac 1 2, p>2, 1\leq r\leq\infty\}\ or\ \{s>\frac 1 p, 1\leq p\leq2, 1\leq r\leq\infty\}\ or\ \{s=\frac 1 p, 1\leq p\leq2, r=1\}
	\end{equation}
	Let $n_0\triangleq4\beta_0^2v_0-v_{0xx}$ with $\beta_0\neq0$ be in $B^s_{p,r}$. Then \eqref{Npeakon} has a unique maximal solution $n$ in $E^s_{p,r}(T^*_-)$ and $n$ depends continuously on the initial data $n_0$. 
	
	Moreover, there exists a positive constant $C$, depending only on $s,p,r,\beta_0\ s.t$
	$$T^*\geq\frac C {\|n_0\|_{B^s_{p,r}}}$$
\end{theo}
\begin{proof}\
	Now  we prove Theorem \ref{well posedness} as the following steps.
	
	\noindent\textbf{Step 1: Constrcuting Approxiamte Solitions and Uniform Bounds.}
	
	Denoting $v^0=0$, we define a sequence $(n^k)_{k\in\mathbb{N}}$ of smooth functions by solving the following linear transport equation:
		$$
		(T_k)	\left\{
			\begin{array}{ll}
				n_t^{k+1}-4(v_x^k+2\beta_0v^k)n_x^{k+1}=16\beta_0^2v^kn^k-4(n^k)^2+8\beta_0v_x^kn^k\\
				v^{k+1}=(4\beta_0^2-\partial_x^2)^{-1}n^{k+1}\\
				v^{k+1}|_{t=0}=S_{k+1}v_0
			\end{array}
			\right.
		$$
	By induction, we assume that $(n^k)_{k\in\mathbb{N}},\forall T>0$. Since $(4\beta_0^2-\partial_x^2)^{-1}$ is an $S^{-2}$-multiplier and 
	$$v^{k+1}=(4\beta_0^2-\partial_x^2)^{-1}n^{k+1}=\frac 1{4|\beta_0|}e^{-|2\beta_0\cdot|\ast n^{k+1}}$$
	also noticing that $B^s_{p,r}$ is an algebra when $(s,p,r)$ satisfies the condition \eqref{index} and the embedding $B^s_{p,r}\hookrightarrow B^{s-1}_{p,r}\hookrightarrow B^{s-2}_{p,r}$, we obtain
	$$\|16\beta_0^2v^kn^k-4(n^k)^2+8\beta_0v_x^kn^k\|_{B^s_{p,r}}\leq C\|n^k\|^2_{B^s_{p,r}}$$
	At the same time $\|S_{k+1}n_0\|_{B^s_{p,r}}\leq C\|n_0\|_{B^s_{p,r}}$, and denote that
	$$\int_0^t\|4(v_x^k+2\beta_0v^k)_x(\tau)\|_{B^s_{p,r}}{\rm d}\tau\leq CN^k(t)\triangleq\int_0^t\|n^k(\tau)\|_{B^s_{p,r}}{\rm d}\tau$$
	Applying Thm\ref{priori estimate}, we have
	\begin{equation}\label{bound1}
		\|n^{k+1}(t)\|_{B^s_{p,r}}\leq e^{CN^k(t)}(\|n_0\|_{B^s_{p,r}}+C\int_0^t e^{-CN^k(\tau)}\|n^k(\tau)\|_{B^s_{p,r}}^2{\rm d}\tau)
	\end{equation}
	
	We may assume $C\geq 1$ and fix a $T>0\ s.t.$
	\begin{equation}
		2C^2\|n_0\|_{B^s_{p,r}}T<1
	\end{equation}
	Suppose that $\forall t\in[0,T]$, $n^k$ satisfies
	\begin{equation}\label{bound2}
		\|n^k(t)\|_{B^s_{p,r}}\leq\frac{C\|n_0\|_{B^s_{p,r}}}{1-2C^2\|n_0\|_{B^s_{p,r}}t}
	\end{equation}
	Note that
	$$e^{CN^k(t_2)-CN^k(t1)}=e^{C\int_{t_1}^{t_2}\|v^k(\tau)\|_{B^s_{p,r}}{\rm d}\tau}\leq e^{\int_{t_1}^{t_2}\frac{C^2\|n_0\|_{B^s_{p,r}}}{1-2C^2\tau\|n_0\|_{B^s_{p,r}}}{\rm d}\tau}=\left(\frac{1-2C^2\|n_0\|_{B^s_{p,r}}t_1}{1-2C^2\|n_0\|_{B^s_{p,r}}t_2}\right)^{1/2}$$
	Plugging \eqref{bound2} into \eqref{bound1} yields
	\begin{equation}
		\begin{array}{ll}
			&\|n^{k+1}(t)\|_{B^s_{p,r}}\\
			&\leq(1-2C^2\|n_0\|_{B^s_{p,r}}t)^{1/2}\left(C\|n_0\|_{B^s_{p,r}}+\int_0^t C(1-2C^2\|n_0\|_{B^s_{p,r}}\tau)^{1/2}\frac{C^2\|n_0\|_{B^s_{p,r}}}{(1-2C^2\|n_0\|_{B^s_{p,r}}\tau)^2}{\rm d}\tau\right)\\
			&\leq\frac{C\|n_0\|_{B^s_{p,r}}}{1-2C^2\|n_0\|_{B^s_{p,r}}t}
		\end{array}
	\end{equation}
	Therefore, $(n^k)_{k\in\mathbb{N}}$ is uniformly bounded in $C([0,T];B^s_{p,r})$. According to the system $(T_k)$ and Lemma\ref{priori estimate} and Lemma\ref{refine priori}, we can deduce that $\{n^k\}_{k\in\mathbb{N}}$ is uniformly bounded in $E^s_{p,r}(T)$.

	\noindent\textbf{Step 2: Convergence and Regularity.}
	
	We will claim that $\{n^k\}_{k\in\mathbb{N}}$ is a Cauchy sequence in $C([0,T];B^{s-1}_{p,r})$, if $\{s\geq\frac 1 2, p>2, 1\leq r\leq\infty\}$ or $\{s>\frac 1 p, 1\leq p\leq2, 1\leq r\leq\infty\}$; or of $C([0,T];B^{\frac 1 p-1}_{p,\infty})$ if $\{1\leq p\leq 2,s=\frac 1 p,r=1\}$.
	
	Indeed, for all $k,p\in\mathbb{R}$, taking the difference between the system $(T_{k+p})$ and ${T_k}$, and denote $w^{k,p}=n^{k+p}-n^{k}$, we have
	\begin{equation}\label{difference}
		\left\{
		\begin{array}{ll}
			(\partial_t-4(v_x^{k+p}+2\beta_0v^{k+p}))\partial_xw^{k+1,p}\\
			=-4(v_x^{k+p}-v_x^k)n_x^{k+1}-8\beta_0(v^{k+p}-v^k)n_x^{k+1}-4(n^{k+p}-n^k)(n^{k+p}+n^k)\\
			\ \ \  -16\beta_0^2\left((v^{k+p}-v^k)n^{k+p}+v^k(n^{k+p}-n^k)\right)+8\beta\left((v_x^{k+p}-v_x^k)n_x^{k+p}+(n^{k+p}-n^k)v_x^k\right)\\
			w^{k+1,p}|_{t=0}=S_{k+p+1}n_0-S_{k+1}n_0
		\end{array}
		\right.
	\end{equation}
	\noindent \textbf{Case1:} $s\geq\frac 1 2, p>2, 1\leq r\leq\infty$ or $s>\frac 1 p, 1\leq p\leq2, 1\leq r\leq\infty$
	
	First, estimate the left-hand side of system \eqref{difference} in $B^{s-1}_{p,r}$. In fact, by Lemma\ref{Bony} and Lemma\ref{multiplier}, we can see
	$$
	\begin{aligned}
		\|(v_x^{k+p}-v_x^k)n_x^{k+1}\|_{B^{s-1}_{p,r}}
		\leq& C \|v_x^{k+p}-v_x^{k}\|_{L^\infty}\|n_x^{k+1}\|_{B^{s-1}_{p,r}}+C\|v_x^{k+p}-v_x^{k}\|_{B^s_{p,r}}\|n_x^{k+1}\|_{B^{-1}_{\infty,\infty}}\\
		&+C\|v_x^{k+p}-v_x^{k}\|_{B^s_{p,r}}\|n_x^{k+1}\|_{B^{\frac 1 p-1}_{p,r}}
		\leq C\|n^{k+p}-n^k\|_{B^{s-1}_{p,r}}\|n^{k+1}\|_{B^s_{p,r}}
	\end{aligned}$$
	$$
	\begin{aligned}
		\|(v^{k+p}-v^k)n_x^{k+1}\|_{B^{s-1}_{p,r}}
		\leq& C \|v^{k+p}-v^{k}\|_{L^\infty}\|n_x^{k+1}\|_{B^{s-1}_{p,r}}+C\|v^{k+p}-v^{k}\|_{B^s_{p,r}}\|n_x^{k+1}\|_{B^{-1}_{\infty,\infty}}\\
		&+C\|v^{k+p}-v^{k}\|_{B^s_{p,r}}\|n_x^{k+1}\|_{B^{\frac 1 p-1}_{p,r}}
		\leq C\|n^{k+p}-n^k\|_{B^{s-1}_{p,r}}\|n^{k+1}\|_{B^s_{p,r}}
	\end{aligned}$$
	$$
	\begin{aligned}
		\|(n^{k+p}-n^k)(n^{k+p}+n^k)\|_{B^{s-1}_{p,r}}
		\leq& C \|n^{k+p}+n^{k}\|_{L^\infty}\|n^{k+p}-n^{k}\|_{B^{s-1}_{p,r}}+C\|n^{k+p}+n^{k}\|_{B^s_{p,r}}\|n^{k+p}-n^k\|_{B^{-1}_{\infty,\infty}}\\
		&+C\|n^{k+p}+n^{k}\|_{B^s_{p,r}}\|n^{k+p}-n^k\|_{B^{\frac 1 p-1}_{p,r}}\\
		\leq& C\|n^{k+p}-n^k\|_{B^{s-1}_{p,r}}(\|n^{k+p}\|_{B^s_{p,r}}+\|n^{k}\|_{B^s_{p,r}})
	\end{aligned}$$
	$$
	\begin{aligned}
		\|(v^{k+p}-v^k)n^{k+p}\|_{B^{s-1}_{p,r}}
		\leq& C \|v^{k+p}-v^{k}\|_{L^\infty}\|n^{k+p}\|_{B^{s-1}_{p,r}}+C\|v^{k+p}-v^{k}\|_{B^s_{p,r}}\|n^{k+p}\|_{B^{-1}_{\infty,\infty}}\\
		&+C\|v^{k+p}-v^{k}\|_{B^s_{p,r}}\|n^{k+p}\|_{B^{\frac 1 p-1}_{p,r}}
		\leq C\|n^{k+p}-n^k\|_{B^{s-1}_{p,r}}\|n^{k}\|_{B^s_{p,r}}
	\end{aligned}$$
	$$
	\begin{aligned}
		\|v^k(n^{k+p}-n^k)\|_{B^{s-1}_{p,r}}
		\leq& C \|v^{k}\|_{L^\infty}\|n^{k+p}-n^{k}\|_{B^{s-1}_{p,r}}+C\|v^{k}\|_{B^s_{p,r}}\|n^{k+p}-n^k\|_{B^{-1}_{\infty,\infty}}\\
		&+C\|v^{k}\|_{B^s_{p,r}}\|n^{k+p}-n^k\|_{B^{\frac 1 p-1}_{p,r}}
		\leq C\|n^{k+p}-n^k\|_{B^{s-1}_{p,r}}\|n^{k}\|_{B^s_{p,r}}
	\end{aligned}$$
	$$
	\begin{aligned}
		\|(v_x^{k+p}-v_x^k)n^{k+p}\|_{B^{s-1}_{p,r}}
		\leq& C \|v_x^{k+p}-v_x^{k}\|_{B^s_{p,r}}\|n^{k+p}\|_{B^{s}_{p,r}}
		\leq C\|n^{k+p}-n^k\|_{B^{s-1}_{p,r}}\|n^{k+p}\|_{B^s_{p,r}}
	\end{aligned}$$
	$$
	\begin{aligned}
		\|(n^{k+p}-n^k)v_x^k\|_{B^{s-1}_{p,r}}
		\leq& C \|v_x^k\|_{L^\infty}\|n^{k+p}-n^{k}\|_{B^{s-1}_{p,r}}+C\|v_x^k\|_{B^{s}_{p,r}}\|n^{k+p}-n^k\|_{B^{-1}_{\infty,\infty}}\\
		&+C\|v_x^k\|_{B^s_{p,r}}\|n^{k+p}-n^k\|_{B^{\frac 1 p-1}_{p,r}}
		\leq C\|n^{k+p}-n^k\|_{B^{s-1}_{p,r}}\|n^{k}\|_{B^s_{p,r}}
	\end{aligned}$$
	Also we can obtain 
	$$\begin{aligned}
		\|S_{k+p+1}n_0-S_{k+1}n_0\|_{B^{s-1}_{p,r}}\leq C\left(\sum_{l=k}^{k+p}2^{-lr}2^{lsr}\|\Delta_ln_0\|_{L^p}^r\right)^{\frac 1 r}\longrightarrow 0(k\rightarrow\infty)
	\end{aligned}$$
	
	Now, recalling \eqref{bound2} and the uniform bound estimates of $\|n^k\|_{B^s_{p,r}}$, and making use of Lemma\ref{priori estimate} or Lemma\ref{refine priori} to the system \eqref{difference}, we have
	$$\|w^{k+1,p}\|_{B^{s-1}_{p,r}}\leq C_T(2^{-k}+\int_0^t\|w^{k,p}(\tau)\|_{B^{s-1}_{p,r}}{\rm d}\tau)$$
	where $C_T=C(s,p,r,T,\beta_0,\|v_0\|_{B^s_{p,r}})$ is a constant, then by induction as $k$ tends to $\infty$
	$$\begin{aligned}
		\|w^{k+1,p}\|_{B^{s-1}_{p,r}}&\leq \sum_{l=0}^p\|w^{k+l+1,1}\|_{B^{s-1}_{p,r}}\leq\sum_{l=0}^pC_T(2^{-(k+l)}+\int_0^tC_T(2^{-(k+l-1)}+\int_0^{t_1}\|w^{k+l,1}\|_{B^{s-1}_{p,r}}{\rm d}t_2){\rm d}t_1)\\
		&\leq\sum_{l=0}^pC_t(2^{-(k+l)}\sum_{m=0}^{k+l}\frac{(2TC_T)^m}{m!}+(C_T)^{k+l+1}\int_0^t\frac{(t-\tau)^{k+l+1}}{(k+l+1)!}{\rm d}\tau)\\
		&\leq\sum_{l=0}^p\left(C_T2^{-(k+l)}e^{2TC_T}+C_T\frac{(TC_T)^{k+l+1}}{(k+l+1)!}\right)\leq C_T2^{-k+1}e^{2TC_T}+\sum_{m=k+1}^{k+p+1}C_T\frac{(TC_T)^m}{m!}\longrightarrow0
	\end{aligned}$$
	which means $(n^k)_{k\in\mathbb{N}}\subset E^s_{p,r}(T)$, and by Fatou's property, we obtain that $n\in L^\infty([0,T];B^s_{p,r})$.
	
	Now for any test function $\phi\in C([0,T];\mathcal{S})$ in the system $(T_k)$, applying Prop\ref{dual}, and taking the limits, it is not hard to check that $n$ is a solution to the system \eqref{Npeakon}.
	
	According to Lemma\ref{existence}, $n\in C([0,T];B^s_{p,r})$, and by system\eqref{Npeakon}, we can get that $n_t\in C([0,T];B^{s-1}_{p,r})$, thus $n\in E^s_{p,r}(T)$.
	
	\noindent \textbf{Case2:} $s=1/p,1\leq p\leq 2$ and $r=1$
	
	Since the critical index $s-1=-(1-\frac 1 p)=-\min(\frac{1 p},\frac{1{p^\prime}})$, we can only apply Lemma\ref{priori estimate} in the stage space $B^{\frac 1 p-1}_{p,\infty}$. And noticing that $B^{\frac 1 p}_{p,1}\hookrightarrow B^{\frac 1 p-1}_{p,1}\hookrightarrow B^{\frac 1 p-1}_{p,\infty}$, and applying Lemma\ref{Bony} to the system \eqref{difference}
	
	$$
	\begin{aligned}
		\|(v_x^{k+p}-v_x^k)n_x^{k+1}\|_{B^{\frac1 p -1}_{p,\infty}}
		\leq& C \|v_x^{k+p}-v_x^{k}\|_{L^\infty}\|n_x^{k+1}\|_{B^{\frac1 p -1}_{p,\infty}}+C\|v_x^{k+p}-v_x^{k}\|_{B^{\frac 1 p}_{p,\infty}}\|n_x^{k+1}\|_{B^{-1}_{\infty,\infty}}\\
		&+C\|v_x^{k+p}-v_x^{k}\|_{B^{\frac 1 p}_{p,1}}\|n_x^{k+1}\|_{B^{-\frac 1 p}_{p^\prime,\infty}}
		\leq C\|n^{k+p}-n^k\|_{B^{\frac1 p -1}_{p,1}}\|n^{k+1}\|_{B^{\frac 1 p}_{p,1}}
	\end{aligned}$$
	$$
	\begin{aligned}
		\|(v^{k+p}-v^k)n_x^{k+1}\|_{B^{\frac1 p -1}_{p,\infty}}
		\leq& C \|v^{k+p}-v^{k}\|_{L^\infty}\|n_x^{k+1}\|_{B^{\frac 1 p -1}_{p,\infty}}+C\|v^{k+p}-v^{k}\|_{B^{\frac 1 p}_{p,\infty}}\|n_x^{k+1}\|_{B^{-1}_{\infty,\infty}}\\
		&+C\|v^{k+p}-v^{k}\|_{B^{\frac 1 p}_{p,1}}\|n_x^{k+1}\|_{B^{-\frac 1 p}_{p^\prime,\infty}}
		\leq C\|n^{k+p}-n^k\|_{B^{\frac 1 p -1}_{p,1}}\|n^{k+1}\|_{B^{\frac 1 p}_{p,1}}
	\end{aligned}$$
	$$
	\begin{aligned}
		\|(n^{k+p}-n^k)(n^{k+p}+n^k)\|_{B^{\frac1 p -1}_{p,\infty}}
		\leq& C \|n^{k+p}+n^{k}\|_{L^\infty}\|n^{k+p}-n^{k}\|_{B^{\frac 1 p}_{p,\infty}}+C\|n^{k+p}+n^{k}\|_{B^{\frac 1 p -1}_{p,\infty}}\|n^{k+p}-n^k\|_{B^{-1}_{\infty,\infty}}\\
		&+C\|n^{k+p}+n^{k}\|_{B^{-\frac 1 p +1}_{p^\prime,\infty}}\|n^{k+p}-n^k\|_{B^{\frac 1 p-1}_{p,1}}\\
		\leq& C\|n^{k+p}-n^k\|_{B^{\frac 1 p -1}_{p,1}}(\|n^{k+p}\|_{B^{\frac 1 p}_{p,1}}+\|n^{k}\|_{B^{\frac 1 p}_{p,1}})
	\end{aligned}$$
	$$
	\begin{aligned}
		\|(v^{k+p}-v^k)n^{k+p}\|_{B^{\frac1 p -1}_{p,\infty}}
		\leq& C \|v^{k+p}-v^{k}\|_{L^\infty}\|n^{k+p}\|_{B^{\frac1 p -1}_{p,\infty}}+C\|v^{k+p}-v^{k}\|_{B^{\frac1 p}_{p,\infty}}\|n^{k+p}\|_{B^{-1}_{\infty,\infty}}\\
		&+C\|v^{k+p}-v^{k}\|_{B^{\frac1 p }_{p,1}}\|n^{k+p}\|_{B^{-\frac 1 p}_{p^\prime,\infty}}
		\leq C\|n^{k+p}-n^k\|_{B^{\frac1 p -1}_{p,1}}\|n^{k}\|_{B^{\frac1 p }_{p,1}}
	\end{aligned}$$
	$$
	\begin{aligned}
		\|v^k(n^{k+p}-n^k)\|_{B^{\frac1 p -1}_{p,\infty}}
		\leq& C \|v^{k}\|_{L^\infty}\|n^{k+p}-n^{k}\|_{B^{\frac1 p -1}_{p,\infty}}+C\|v^{k}\|_{B^{\frac1 p }_{p,\infty}}\|n^{k+p}-n^k\|_{B^{-1}_{\infty,\infty}}\\
		&+C\|v^{k}\|_{B^{-\frac 1 p +1}_{p^\prime,\infty}}\|n^{k+p}-n^k\|_{B^{\frac 1 p-1}_{p,1}}
		\leq C\|n^{k+p}-n^k\|_{B^{\frac1 p -1}_{p,1}}\|n^{k}\|_{B^{\frac 1 p}_{p,1}}
	\end{aligned}$$
	$$
	\begin{aligned}
		\|(v_x^{k+p}-v_x^k)n^{k+p}\|_{B^{\frac1 p -1}_{p,\infty}}
		\leq& C \|v_x^{k+p}-v_x^{k}\|_{B^{\frac 1 p}_{p,1}}\|n^{k+p}\|_{B^{\frac 1 p}_{p,1}}
		\leq C\|n^{k+p}-n^k\|_{B^{\frac 1 p -1}_{p,1}}\|n^{k+p}\|_{B^{\frac 1 p}_{p,1}}
	\end{aligned}$$
	$$
	\begin{aligned}
		\|(n^{k+p}-n^k)v_x^k\|_{B^{\frac1 p -1}_{p,\infty}}
		\leq& C \|v_x^k\|_{B^{\frac 1 p}_{p,1}}\|n^{k+p}-n^{k}\|_{B^{\frac1 p -1}_{p,\infty}}+C\|v_x^k\|_{B^{\frac 1 p}_{p,\infty}}\|n^{k+p}-n^k\|_{B^{-1}_{\infty,\infty}}\\
		&+C\|v_x^k\|_{B^{\frac1 {p^\prime} }_{p^\prime,\infty}}\|n^{k+p}-n^k\|_{B^{\frac 1 p-1}_{p,1}}
		\leq C\|n^{k+p}-n^k\|_{B^{\frac 1 p-1}_{p,1}}\|n^{k}\|_{B^{\frac 1 p}_{p,1}}
	\end{aligned}$$
	Also we can obtain the initial data of system \eqref{difference} satisfies
	$$\begin{aligned}
		\|S_{k+p+1}n_0-S_{k+1}n_0\|_{B^{\frac1 p -1}_{p,\infty}}\leq  C\sum_{l=k}^{k+p+1}2^{l(1/p-1)}\|\Delta_ln_0\|_{L^p}\leq C2^{-k}\|n_0\|_{B^{\frac 1 p -1}_{p,1}}\longrightarrow 0(k\rightarrow\infty)
	\end{aligned}$$
	Since $B^{\frac 1 p}_{p,1}\hookrightarrow B^{\frac 1 p}_{p,\infty}$ and $B^{\frac 1 p}_{p,1}\hookrightarrow L^\infty$, we obtain the estimate of the flow
	$$\int_0^t\|4\partial_{x}(v_x^{k+p}+2\beta_0v^{k+p})\|_{B^{1/p}_{p,\infty}\cap L^\infty}{\rm d}\tau\leq C\int_0^t\|n^{k+p}\|_{B^{1/p}_{p,1}}{\rm d}\tau$$
	According to the logarithm interpolation of Lemma\ref{Besov}, and the uniform bound of $\{n^k\}_{k\in\mathbb{N}}$, we obtain
	$$\|w^{k+1,p}\|_{B^{\frac{1}{p}-1}_{p,\infty}}\leq C_T\left(2^{-k}+\int_0^t\|w^{k,p}\|_{B^{\frac 1 p -1}_{p,\infty}}\log\left(e+\frac 1 {\|w^{k,p}\|_{B^{\frac 1 p -1}_{p,\infty}}}\right){\rm d}\tau\right)$$
	denote that $w_k(t)=\sup_{\tau\in[0,t],p\in\mathbb{N}}\|(n^{k+p}-n^{k})(\tau)\|_{B^{\frac 1 p -1}_{p,\infty}}$, then we have
	$$w_{k+1}(t)\leq C_T\left(2^{-k}+\int_0^tw_k(\tau)log\left(e+\frac1{w_k(\tau)}\right){\rm d}\tau\right)$$
	
	Note that $\left\{w_k\right\}_{k\in\mathbb{N}}$, thus it is obvious that $\mu(x)=C_Txlog(e+\frac 1 x)$ is an Osgood modulus of continuity. Applying Lesbesgue-Fatou's lemma in real analysis, we conclude that, for $\forall t\in[0,T]$
	$$\limsup_{k\rightarrow\infty}w_k(t)\leq\int_0^t\mu(\limsup_{k\rightarrow\infty}w_k(\tau)){\rm d}\tau$$
	Hence according to the Osgood lemma, we have
	$$\limsup_{k\rightarrow\infty} w_k(t)=0,\ t\in[0,T]$$
	which implies that ${n^k}_{k\in\}_{k\in\mathbb{N}}}$ is a Cauchy sequence in the space $C([0,T];B^{1/p-1}_{p,\infty})$ and converges to the limit function $n\in C([0,T];B^{1/p}_{p,\infty})$.Along the same line of Case1, similarly we deduce that $n$ is a solution to \eqref{Npeakon} and $n\in E^{\frac 1 p}_{p,1}(T)$.
	
	\noindent\textbf{Step 3: Uniqueness.}
	
	Suppose $n_1,n_2$ are two solutions of \eqref{Npeakon} with initial data $n_{10}$ and $n_{20}$ respectively. Denoting $w\triangleq n_1-n_2$, we have
	\begin{equation}\label{unique equation}
		\left\{
		\begin{array}{ll}
			(\partial_t-4(v_{1x}+2\beta_0v_1))\partial_xw\\
			=-4(v_{1x}-v_{2x})n_{2x}-8\beta_0(v_1-v_2)n_{2x}-4(n_1-n_2)(n_1+n_2)\\
			\ \ \  -16\beta_0^2\left((v_1-v_2)n_1+v_2(n_1-n_2)\right)+8\beta\left((v_{1x}-v_{2x})n_1+(n_1-n_2)v_{1x}\right)\\
			v_1=(4\beta_0^2-\partial_{x}^2)^{-1}n_1,\ v_2=(4\beta_0^2-\partial_{x}^2)^{-1}n_2\\
			w|_{t=0}=n_{10}-n_{20}
		\end{array}
		\right.
	\end{equation}
	
	If $s\geq 1/2, p>2$ or $s>1/p, 1\leq p\leq2$, along the same computation in step 2, $\forall t\in[0,T]$, we have inequality
	$$
	e^{-C\int_0^t\|n_1\|_{B^s_{p,r}}{\rm d}\tau}\|w(t)\|_{B^{s-1}_{p,r}}\leq\|w_0\|_{B^{s-1}_{p,r}}+C\int_0^te^{-C\int_0^{t^\prime}\|n10\|_{B^s_{p,r}}{\rm d}\tau}\|w(t^\prime)\|_{B^{s-1}_{p,r}}(\|n_1\|_{B^s_{p,r}}+\|n_2\|_{B^s_{p,r}}){\rm d}t^\prime
	$$
	By Gronwall's inequality
	\begin{equation}\label{unique esimate1}
		\|w(t)\|_{B^{s-1}_{p,r}}\leq\|w_0\|_{B^{s-1}_{p,r}}\exp\left(C\int_0^t\|n_1\|_{B^s_{p,r}}+\|n_2\|_{B^s_{p,r}}{\rm d}\tau\right)
	\end{equation}
	
	If $s=1/p,1\leq p\leq 2,r=1$, similar to step 2, we have inequality
	
	$
	\begin{aligned}
		&e^{-C\int_0^t\|n_1\|_{B^{1/p}_{p,1}}{\rm d}\tau}\|w(t)\|_{B^{\frac 1 p -1}_{p,\infty}}\\
		&\leq\|w_0\|_{B^{\frac 1 p -1}_{p,\infty}}+C\int_0^te^{-C\int_0^{t^\prime}\|n_1\|_{B^{\frac 1 p}_{p,1}}{\rm d}\tau}\|w(t^\prime)\|_{B^{\frac 1 p -1}_{p,\infty}}(\|n_1\|_{B^{\frac 1 p}_{p,1}}+\|n_2\|_{B^{\frac 1 p}_{p,1}})\log\left(e+\frac{\|n_1\|_{B^{\frac 1 p}_{p,1}}+\|n_2\|_{B^{\frac 1 p}_{p,1}}}{\|w(t)\|_{B^{\frac 1 p -1}_{p,\infty}}}\right){\rm d}t^\prime\nonumber
	\end{aligned}$
	denote $h(t)\triangleq \|w(t)\|_{B^{1/p-1}_{p,\infty}}e^{-C\int_0^t\|n_1\|_{B^{1/p}_{p,1}}{\rm d}\tau}/\left(\|n_1\|_{L^\infty_T(B^{1/p}_{p,1})}+\|n_2\|_{L^\infty_T(B^{1/p}_{p,1})}\right)$ and $|h(t)|\leq1$, obviously, since
	
	$$x\log(e+\frac 1 x)\leq x(2-\log x),\ \forall x\in[0,1]$$
	then we have
	
	$$h(t)\leq h(0)+C\int_0^th(\tau)(\|n_1\|_{L^\infty_T(B^{1/p}_{p,1})}+\|n_2\|_{L^\infty_T(B^{1/p}_{p,1})})^2(2-\log h(\tau)){\rm d}\tau$$
	
	Since $\mu(x)=x(2-\log x)$ is an Osgood modulus of continuity on $[0,1]$, then by Osgood lemma, we get inequality
	
	$$\int_{h(0)}^{h(t)}\frac{{\rm d}r}{r(2-\log r)}\leq C\int_0^t(\|n_1\|_{L^\infty_T(B^{1/p}_{p,1})}+\|n_2\|_{L^\infty_T(B^{1/p}_{p,1})})^2{\rm d}\tau$$
	then by the definition of $h$, we have
	\begin{equation}\label{unique esimate2}
		\begin{aligned}
			&\|n_1-n_2\|_{L^\infty_T(B^{1/p-1}_{p,1})}\\
			&\leq\  e^2(\|n_1\|_{L^\infty_T(B^{1/p}_{p,1})}+\|n_2\|_{L^\infty_T(B^{1/p}_{p,1})})\exp(C\int_0^t \|n_1\|_{L^\infty_T(B^{1/p}_{p,1})}+\|n_2\|_{L^\infty_T(B^{1/p}_{p,1})}{\rm d}\tau)\\
			&\ \ \ \ \times(\|n_{10}-n_{20}\|_{(B^{1/p-1}_{p,\infty})})^{\exp(-C\int_0^t\|n_1\|_{L^\infty_T(B^{1/p}_{p,1})}+\|n_2\|_{L^\infty_T(B^{1/p}_{p,1})}{\rm d}\tau)}
		\end{aligned}
	\end{equation}
	Then the uniqueness of solution $n$ to the system \eqref{Npeakon} in the space $E^p_{p,r}(T)$	follows from \eqref{unique esimate1} and \eqref{unique esimate2}.
			
	\noindent\textbf{Step 4: Continuous Dependence.}
	
	Now according to Lemma\ref{continuous}, denote $\bar{\mathbb{N}}=\mathbb{N}\cup\{\infty\}$. suppose $n^k\in C([0,T];B^s_{p,r})$ is the solution to \eqref{Npeakon} with initial data $n_0^k\in B^s_{p,r}$, then we have a sequence of systems
	
	$$
	\left\{
	\begin{array}{l}
		\partial_t n^k-4(v_x^k+2\beta_0v^k)\partial_{x}n^k=f^k\\
		n^k=4\beta_0^2 v^k-v_{xx}^k\\
		n^k|_{t=0}=n_0^k
	\end{array}
	\right.
	$$
	with $f^k\triangleq16\beta_0^2v^kn^k-4(n^k)^2+8\beta v_x^kn^k$.
	
	Next we claim that if $n^k_0\rightarrow n^\infty_0$ in $B^s_{p,r}$ when $k$ tends to $\infty$, then $n^k\rightarrow n^\infty$ in $C([0,T];B^s_{p,r})$ with $T$ satisfying $2C^2\sup_{k\in\bar{\mathbb{N}}}\|n_0^k\|_{B^s_{p,r}}T<1$.
	
	\noindent(1) r<$\infty$. Decompose $n^k$ into $n^k=y^k+z^k$, and satisfy
	\begin{equation}
		\left\{
		\begin{array}{l}
			\partial_t y^k-4(v^k_x+2\beta_0v^k)\partial_x y^k=f^\infty\\
			y^k|_{t=0}=n^k_0\nonumber
		\end{array}
		\right.
		\quad {\rm and} \quad
		\left\{
		\begin{array}{l}
			\partial_t z^k-4(v^k_x+2\beta_0v^k)\partial_x z^k=f^k-f^\infty\\
			z^k|_{t=0}=n^k_0-n^\infty_0\nonumber
		\end{array}
		\right.
	\end{equation}
	Denote $m=\sup_{k\in\bar{\mathbb{N}}}\|n_0^k\|_{B^s_{p,r}}$, and letting $M=\frac{Cm}{1-2C^2mT}$, according to the proof of the existence, we get that
	$$\|n^k(t)\|_{B^s_{p,r}}\leq M\quad \forall t\in[0,T]$$
	
	Hence, $(f^)k_{k\in\}_{k\in\mathbb{N}}}$ is uniformly bounded in $C([0,T];B^s_{p,r})$ and
	$$\|v^k_x+2\beta_{0}v^k\|_{B^{s+1}_{p,r}}\leq C\|v^k\|_{B^{s+2}_{p,r}}\leq C\|n^k\|_{B^s_{p,r}}\leq CM$$
	also it is obvious that 
	$$\|(v^k-v^\infty)_x+2\beta_{0}(v^k-v^\infty)\|_{B^{s+1}_{p,r}}\leq C\|n^k-n^\infty\|_{B^s_{p,r}}$$
	similar to the proof of the uniqueness, since $n_0^k\rightarrow n_0^\infty$ in $B^s_{p,r}$, then we have $n^k\rightarrow n^\infty$ in $L^1([0,T];B^{s-1}_{p,r})\Rightarrow v^k_x+2\beta_{0}v^k\rightarrow v^\infty_x+2\beta_{0}v^\infty$ in $L^1([0,T];B^{s}_{p,r})$. Then by Lemma\ref{continuous}
	
	\begin{center}
		$y^k\rightarrow n^\infty$ in $C([0,T];B^s_{p,r})$ as $k\rightarrow\infty$ 
	\end{center}
	
	To control $z^k$, we need to estimate $f^k-f^\infty$, by Lemma\ref{Bony}, we have
	$$\|f^k-f^\infty\|_{B^s_{p,r}}\leq C(\|n^k\|_{B^s_{p,r}}+\|n^\infty\|_{B^s_{p,r}})\|n^k-n^\infty\|_{B^s_{p,r}}$$
	then according to Lemma\ref{priori estimate}
	$$\|z^k\|_{B^s_{p,r}}\leq Ce^{CMT}(\|n^k_0-n^\infty_0\|_{B^s_{p,r}}+\int_0^t\|(n^k-n^\infty)(\tau)\|_{B^s_{p,r}}{\rm d}\tau)$$
	
	Therefore, we have $\forall \varepsilon>0$, $\exists N_1$ large enough, when $k>N_1$, $\|y^k-n^\infty\|_{B^s_{p,r}}<\varepsilon$, then
	$$\|n^k-n^\infty\|_{B^s_{p,r}}\leq\|z^k\|_{B^s_{p,r}}+\|y^k-n^\infty\|_{B^s_{p,r}}\leq\varepsilon+Ce^{CMT}(\|n^k_0-n^\infty_0\|_{B^s_{p,r}}+\int_0^t\|(n^k-n^\infty)(\tau)\|_{B^s_{p,r}}{\rm d}\tau)$$
	by Gronwall's inequality, $\forall t\in[0,T]$
	$$\|n^k-n^\infty\|_{B^s_{p,r}}\leq C(\varepsilon+\|n_0^k-n_0^\infty\|_{B^s_{p,r}})$$
	Hence we gain the continuity of the system \eqref{Npeakon} in $C([0,T];B^s_{p,r})$ with respect to the initial data in $B^s_{p,r}$.
	
	\noindent(2) $r=\infty$. Similar to Step2, we have $\|n^k-n^\infty\|_{L^\infty([0,T];B^{s-1}_{p,\infty})}\rightarrow0(k\rightarrow\infty)$. Then $\forall\phi\in B^{-s}_{p^\prime,1}$, we have 
	$$ \langle n^k-n^\infty\rangle=\langle S_j(n^k-n^\infty),\phi\rangle+\langle (Id-S_j)(n^k-n^\infty),\phi\rangle=\langle n^k-n^\infty,S_j\phi\rangle+\langle n^k-n^\infty,(Id-S_j)\phi\rangle$$
	and
	$$|\langle n^k-n^\infty,(Id-S_j)\phi\rangle|\leq CM\|\phi-S_j\phi\|_{B^{-s}_{p^\prime,1}}\quad{\rm and}\quad|\langle n^k-n^\infty,S_j\phi\rangle|\leq \|n^k-n^\infty\|_{B^{s-1}_{p,\infty}}\|S_j\phi\|_{B^{1-s}_{p^\prime,1}}$$
	by Lemma\ref{Besov}, we can conclude that $\langle n^k-n^\infty,\phi\rangle\rightarrow0$ as $k\rightarrow0$ for $r=\infty$. Thus we gain the continuity of the system \eqref{Npeakon} in $C_w([0,T];B^s_{p,\infty})$ with respect to the initial data in $B^s_{p,\infty}$.

\end{proof}
\section{Blow-up criterion}
We can use the method introduced by Zhang and Yin\cite{blowupcriterion} to obtain a blow-up criterion as follows.
\begin{theo}\label{blow up criterion}
	Let $n_0\in B^s_{p,r}$ with $1\leq p,r\leq \infty$, and $(s,p,r)$ satisfy condition \eqref{index} and let $T>0$ be the maximal existence time of the corresponding solution $n$ to \eqref{Npeakon}. If $T<\infty$, then we have
	$$\int_0^T \|n(t)\|_{L^\infty}{\rm d}t$$
\end{theo}
\begin{proof}
	For simplicity, consider the case $1<p<\infty$. For all $\sigma$ satisfies $0<\sigma\leq s$, according to the product laws, since $\|v\|_{L^\infty}\leq C\|n\|_{L^\infty}$ and $\|v\|_{B^\sigma_{p,r}}\leq C\|n\|_{B^\sigma_{p,r}}$, we have
	\begin{equation}\label{transport term}
		\|16\beta_0^2vn-4n^2+8\beta_0v_xn\|_{B^\sigma_{p,r}}\leq C\|n\|_{L^\infty}\|n\|_{B^\sigma_{p,r}}
	\end{equation}
	
	\noindent(1) $\sigma\in(0,1)$.
	
	Applying Lemma\ref{priori estimate} with $p_1=\infty$, then
	$$\|n(t)\|_{B^\sigma_{p,r}}\leq \|n_0\|_{B^\sigma_{p,r}}+C\int_0^t\|n(\tau)\|_{B^\sigma_{p,r}}\|n(\tau)\|_{L^\infty}{\rm d}\tau+\int_0^tCV_{p_1}^\prime\|n(\tau)\|_{B^\sigma_{p,r}}{\rm d}\tau$$
	with
	$$V_{p_1}^\prime(t)=\|4(v_x+2\beta_0 v)_x\|_{B^0_{\infty,\infty}\cap L^\infty}\leq C\|n\|_{L^\infty}$$
	by Gronwall's inequality
	$$\|n(t)\|_{B^\sigma_{p,r}}\leq\|n_0\|_{B^\sigma_{p,r}}\exp(C\int_0^t\|n(\tau)\|_{L^\infty}{\rm d}\tau)$$
	which implies
	\begin{equation}\label{blow0-1}
		\int_0^T\|n(\tau)\|_{L^\infty}{\rm d}\tau<\infty\Longrightarrow\|n\|_{L^\infty_T(B^\sigma_{p,r})}<\infty
	\end{equation}
	\noindent(2) $\sigma=1$
	
	Applying Lemma\ref{priori estimate} with $p=p_1$, 
	and
	$$V^\prime_{p_1}=\|4(v_x+2\beta_0v_x)\|_{B^{1/p}_{p,\infty}\cap L^\infty}\leq C(\|v\|_{L^\infty}+\|v\|_{B^{1/p+2}_{p,\infty}})\leq C(\|n\|_{L^\infty}+\|n\|_{B^{1/p}_{p,r}})$$
	then we have
	$$\|n(t)\|_{B^1_{p,r}}\leq \|n_0\|_{B^1_{p,r}}+C\int_0^t\|n(\tau)\|_{B^1_{p,r}}\|n(\tau)\|_{L^\infty}{\rm d}\tau+C\int_0^t(\|n(\tau)\|_{L^\infty}+\|n(\tau)\|_{B^{1/p}_{p,r}})\|m(\tau)\|_{B^1_{p,r}}{\rm d}\tau$$
	by Gronwall's inequality
	\begin{equation}
		\|n(t)\|_{B^1_{p,r}}\leq C\|n_0\|_{B^1_{p,r}}\exp(C\int_0^t\|n(\tau)\|_{L^\infty}+\|n(\tau)\|_{B^{1/p}_{p,r}}{\rm d}\tau)
	\end{equation}
	which implies 
	\begin{equation}\label{blow1}
		\|n\|_{L^\infty_T(B^{1/p}_{p,r})}<\infty\ \ {\rm and}\ \ \int_0^T\|n(\tau)\|_{L^\infty}{\rm d}\tau<\infty\Longrightarrow\|n\|_{L^\infty_T(B^\sigma_{p,r})}<\infty
	\end{equation}
	
	\noindent(3) $\sigma>1$
	
	According to \eqref{transport term} and Lemma\ref{priori estimate} with $p_1=\infty$, and
	$$V^\prime_{p_1}(t)=\|4(v_x+2\beta_0v_x)\|_{B^{\sigma-1}_{\infty,r}}\leq C\|v\|_{B^{\sigma+1}_{\infty,r}}\leq C\|n\|_{B^{\sigma-1+1/p}_{p,r}}$$
	then we have
	$$\|n(t)\|_{B^\sigma_{p,r}}\leq \|n_0\|_{B^\sigma_{p,r}}+C\int_0^t\|n(\tau)\|_{B^\sigma_{p,r}}\|n(\tau)\|_{L^\infty}{\rm d}\tau+C\int_0^t\|n(\tau)\|_{B^{\sigma-1+1/p}_{p,r}}\|n(\tau)\|_{B^\sigma_{p,r}}{\rm d}\tau$$
	by Gronwall's inequality
	$$\|n(t)\|_{B^\sigma_{p,r}}\leq\|n_0\|_{B^\sigma_{p,r}}\exp(C\int_0^t\|n(\tau)\|_{L^\infty}+\|n(\tau)\|_{B^{\sigma-1+1/p}_{p,r}}{\rm d}\tau)$$
	which implies
	$$\|n\|_{L^\infty_T(B^{\sigma-1+1/p}_{p,r})}<\infty\ \ {\rm and}\ \ \int_0^T\|n(\tau)\|_{L^\infty}{\rm d}\tau<\infty\Longrightarrow\|n\|_{L^\infty_T(B^\sigma_{p,r})}<\infty$$
	
	If $\sigma-1+\frac 1 p>1$, we can repeat the above process. Indeed, since
	$$\|4(v_x+2\beta_0v)_x\|_{B^{\sigma-1+1/p-1}_{\infty,r}}\leq C\|v\|_{B^{\sigma+1/p}_{\infty,r}}\leq C\|n\|_{B^{\sigma-2+2/p}_{p,r}}$$
	we similar have
	$$\|n\|_{L^\infty_T(B^{\sigma-2+1/p}_{p,r})}<\infty\ \ {\rm and}\ \ \int_0^T\|n(\tau)\|_{L^\infty}{\rm d}\tau<\infty\Longrightarrow\|n\|_{L^\infty_T(B^{\sigma-1+1/p}_{p,r})}<\infty$$
	
	By induction, we can choose $k\in\mathbb{N}^+$ such that $0<\sigma-k(1-/frac 1 p)\leq 1$, and $B^1_{p,r}\hookrightarrow B^{\sigma-k(1-\frac 1 p)}_{p,r}$, gives that
	$$\|n\|_{L^\infty_T(B^1_{p,r})}\leq \infty\ \ {\rm and}\ \ \int_0^T\|n\|_{L^\infty}<\infty\Rightarrow\|n\|_{L^\infty_T(B^{\sigma-(k-1)(1-\frac 1 p)}_{p,r})}<\infty\Rightarrow\cdots\Rightarrow\|n\|_{L^\infty_T(B^{\sigma}_{p,r})}<\infty$$
	Therefore we can conclude that
	$$\|n\|_{L^\infty_T(B^{1}_{p,r})}<\infty\ \ {\rm and}\ \ \int_0^T\|n\|_{L^\infty}<\infty\Rightarrow\|n\|_{L^\infty_T(B^{\sigma}_{p,r})}<\infty$$
	
	Return to $L^\infty_T(B^s_{p,r})$ and suppose $\int_0^T\|n\|_{L^\infty}{\rm d}\tau<\infty$, we have:
	
	\noindent\textbf{Case1: s<1} By (1), we have $\|n\|_{L^\infty_T(B^s_{p,r})}<\infty$.
	
	\noindent\textbf{Case2: s=1} By (1), we have $\|n\|_{L^\infty_T(B^{1/p}_{p,r})}<\infty$ and by (2) we get that
	$\|n\|_{L^\infty_T(B^1_{p,r})}<\infty$.
	
	\noindent\textbf{Case3: s>1} By (1), we have $\|n\|_{L^\infty_T(B^{1/p}_{p,r})}<\infty$ and by (2) we get that
	$\|n\|_{L^\infty_T(B^1_{p,r})}<\infty$, lastly by the induction in (3), we can deduce that 	$\|n\|_{L^\infty_T(B^1_{p,r})}<\infty$.
	
	Therefore $\int_0^T\|n\|_{L^\infty}{\rm d}\tau<\infty\Rightarrow\|n\|_{L^\infty(B^s_{p,r})}<\infty$. Then by what we proved in Thm\ref{well posedness}, $\exists T_0$ small enough such that for some suitable $\varepsilon>0$,  $2C^2\|n(T-\varepsilon)\|_{B^s_{p,r}}T_0<1$, then follow then same line of Step1 in the proof of Thm\ref{well posedness}, we can extend the solution $n$ on $[0,T)\times \mathbb{R}$ to $\tilde{n}$ on $[0,T-\varepsilon+T_0)\times\mathbb{R}$ and by the uniqueness $n\equiv\tilde{n}$ for any $t\in[0,T]$.

\end{proof}

\begin{coro}
	Let $n_0\in B^s_{p,r}$ with $(s,p,r)$ satisfying \eqref{index}, $T>0$ be the maximal existence time of the corresponding solution $n\in E^s_{p,r}(T_-)$ to system \eqref{Npeakon}. Then $n$ blows up in finite time if and only if
	$$\limsup_{t\uparrow T}\|n\|_{L^\infty}=\infty$$
\end{coro}

\section{Blow-up}

	As we have the blow-up criterion, consider the following initial value problem
	\begin{equation}\label{flow map}
		\left\{
		\begin{array}{l}
			\psi_t(t,x)=-4(v_x+2\beta_0v)(t,\psi(t,x)),\ t\in[0,T]\\
			\psi(0,x)=x
		\end{array}
		\right.
	\end{equation}
	then the system \eqref{Npeakon} in Lagrangian coordinates can be written as
	\begin{equation}\label{Lagrangian}
		\left\{
		\begin{array}{l}
			\partial_tn(t,\psi(t,x))=(16\beta_0^2vn-4n^2+8\beta_0v_xn)(t,\psi(t,x))\\
			n(0,\psi(0,x))=n_0(x)
		\end{array}
		\right.
	\end{equation}
	
	\begin{theo}
		Let $n_0\in B^s_{p,r}$ with $(s,p,r)$ satisfying \eqref{index}. There exists time $T_1=\frac 1 {32\max(\beta_0^2,\beta_0^{-2})\|v_0\|_{W^{1,\infty}}}$, if  $\exists\ x_0\in\mathbb{R}$ such that
		\begin{equation}\label{blow initial}
			n_0(x_0)<-\sqrt{\frac b 2}+\frac{\sqrt{2b}}{1-\exp(2\sqrt{2b}T_1)}\quad{\rm wtih}\quad b=(16\beta_0^2\|v_0\|_{L^\infty})^2+(8\beta_0\|v_{0x}\|_{L^\infty})^2
		\end{equation}
		then $n$ will blow up within $T_1$.
		
		More precisely, $n$ will blow up at time $T_2=\frac{1}{2\sqrt{2b}}\log(\frac{\sqrt{2}n_0(x_0)-\sqrt{b}}{\sqrt{2}n_0(x_0)+\sqrt{b}})$.
	\end{theo}
	\begin{proof}
		Since $v$ is the solution of  system \eqref{Npeakon2}, and $v_x$ satisfies
		\begin{equation}
			\left\{
			\begin{array}{l}
				\partial_t(v_x)-(4v_x+8\beta_0v)\partial_x(v_x)=8\beta_0(v_x)^2-16\beta_0^2vv_x+8\beta_0\partial_xP_1(D)(2\beta_0^2v^2+v_x^2)+8\beta_0^2P_1(D)(4\beta_0^2v^2-v_x^2)\\
				v_x|_{t=0}=v_{0x}\nonumber
			\end{array}
			\right.
		\end{equation}
		then we have estimates
		\begin{equation}\label{W1,infty}
			\|v\|_{W^{1,\infty}}\leq\|v_0\|_{W^{1,\infty}}+8\max(\beta_0^2,\beta_0^{-2})\int_0^t\|v(\tau)\|_{W^{1,\infty}}^2{\rm d}\tau
		\end{equation}
		
		Assume that $\|v(t)\|_{W^{1,\infty}}\leq 2\|v_0\|_{W^{1,\infty}}$ is satisfied for all $t\in[0,T_1]$, then by inequality \eqref{W1,infty}, we have
		$$\|v(t)\|_{W^{1,\infty}}\leq \|v_0\|_{W^{1,\infty}}+32\max(\beta_0^2,\beta_0^{-2})t\|v_0\|_{W^{1,\infty}}62\leq  2\|v_0\|_{W^{1,\infty}}$$
		Thus we have $\|v(t)\|_{W^{1,\infty}}\leq 2\|v_0\|_{W^{1,\infty}}$ for all $t\in[0,T_1]$.
		
		Consider the system \eqref{Lagrangian} for all $t\in[0,T_1]$, then it is obvious that for all $(t,x_0)\in[0,T_1]\times\mathbb{R}$
		$$|16\beta_0^2 v\cdot n|\leq 32\beta_0^2\|v_0\|_{L^\infty}|n|\leq (16\beta_0^2\|v_0\|_{L^\infty})^2+n^2$$
		and
		$$|8\beta_0v_x\cdot n|\leq 16|\beta_0|\|v_{0x}\|_{L^\infty}|n|\leq (8\beta_0\|v_{0x}\|_{L^\infty})^2+n^2$$
		denote $b=(16\beta_0^2\|v_0\|_{L^\infty})^2+(8\beta_0\|v_{0x}\|_{L^\infty})^2$, then according to \eqref{Lagrangian}, we have
		$$\partial_t n(t,\psi(t,x_0))\leq-2n^2+b$$
		
		Consider the ordinary differential equation
		$$f^\prime(t)=-2f^2+b,f(0)=n_0(x_0)$$
		then we have
		$$n(t,\psi(t,x_0))\leq\int_0^t f^\prime(\tau){\rm d}\tau+n_0(x_0)=\sqrt{\frac b 2}\cdot\frac{1+Ce^{-2\sqrt{2b}t}}{1-Ce^{-2\sqrt{2b}t}}\quad{\rm with}\quad C=\frac{\sqrt{2}n_0(x_0)-\sqrt{b}}{\sqrt{2}n_0(x_0)+\sqrt{b}}$$
		It is obvious that $1-Ce^{-2\sqrt{2b}t}$ is an increasing function, and by \eqref{blow initial}
		$$1-C<0\quad{\rm and}\quad1-Ce^{-2\sqrt{2b}T_1}>0\quad{\rm and}\quad1-Ce^{-2\sqrt{2b}T_2}=0$$
		then we can easily get that $n(T_2,\psi(T_2,x_0))=-\infty$ and $T_2<T_1$.
		
		Therefore $n$ will blow up at time $T_2=\frac{1}{2\sqrt{2b}}\log(\frac{\sqrt{2}n_0(x_0)-\sqrt{b}}{\sqrt{2}n_0(x_0)+\sqrt{b}})$.
	\end{proof}

\section{Existence of global solution}
	Similarly, we consider the initial value problem \eqref{flow map}, applying classical results in the theory of ODEs, one can obtain the following results of $\psi$.
	
	\begin{lemm}\cite{flow}
		Let $v_x+2\beta_0v\in C([0,T];H^s(\mathbb{R}))\cap C^1([0,T];H^{s-1}(\mathbb{R}))$, $s\geq 2$. Then the problem \eqref{flow map} has a unique solution $\psi\in C^1([0,T]\times\mathbb{R})$. Moreover, the map $\psi(t,\cdot)$ is an increasing diffeomorphism of $\mathbb{R}$ with $\psi_x(0,x)=1$ and
		$$\psi_x(t,x)=\exp(\int_0^t-4(v_x+2\beta_0v)_x)(\tau,\psi(\tau,x){\rm d}\tau)>0,\ \ \forall(t,x)\in[0,T)\times \mathbb{R}$$
	\end{lemm}
	
	We then have the following sign-preserved result.
	\begin{lemm}
		Let $n_0\in H^s(\mathbb{R}),\ s>\frac 1 2$ and $n_0(x_0)\geq0,\ \forall x\in\mathbb{R}$. Assume $T_{n_0}$ is the maximal existence time of the solution $n(t,x)$ to \eqref{Npeakon}, then
		$$n(t,x)\geq0,\ \forall(t,x)\in[0,T_{n_0})\times\mathbb{R}$$ 
	\end{lemm}
	\begin{proof}
		From \eqref{Npeakon} we see that $\forall x\in\mathbb{R}$
		$$\begin{aligned}
			\frac {\rm d}{{\rm d}t}(n(t,\psi(t,x))\cdot\psi_x(t,x))&=n_t(t,\psi(t,x))\cdot\psi_x+n_x(t,\psi(t,x))\cdot\psi_t\psi_x+n(t,\psi(t,x))\cdot\psi_{tx}\\
			&=\psi_x\cdot(n_t-4(n(v_x+2\beta_0v))_x)=0
		\end{aligned}$$
		solving the ODE, we have
		$$n(t,\psi(t,x))\cdot\psi_x(t,x)=n_0(x)\cdot\psi_x(0,x)$$
		
		Noticing $n_0\geq0$ and$\psi_x>0$, we get $n(t,\psi(t,x))\geq0$. Since $\psi(t,x)$ is a diffeomorphism on $\mathbb{R}$, we completes the proof.
	\end{proof}
	
	\begin{lemm}\label{sign}
		Let $n_0\in H^s(\mathbb{R}),\ s>\frac 1 2$ and $n_0(x)\geq 0,\ \forall x\in\mathbb{R}$. Assume $T_{n_0}>0$ is the maximal existence time of the solution $n(t,x)$ to \eqref{Npeakon} with initial data $n_0$. Then
		$$v(t,x)\geq 0\quad{\rm and}\quad|v_x(t,x)|\leq 2|\beta_0|v(t,x)\leq\sqrt{2}|\beta_0|\|v\|_{H^1},\ \forall (t,x)\in [0,T_{n_0})\times\mathbb{R}$$
		
	\end{lemm}
	\begin{proof}
		Since we have
		$$v(t,x)=\frac{1}{4|\beta_0|}e^{-|2\beta_0\cdot|}\ast n(t,x)\geq0$$
		and
		$$|v_x(t,x)|=|\frac 1 2(-sgn(\cdot)e^{-|2\beta_0\cdot|}\ast n)(t,x)|\leq2|\beta_0|v(t,x)$$
		
		On the other hand, for almost $\forall x\in\mathbb{R}$, we have
		$$
		\begin{aligned}
			2v^2(t,x)&=\int_{-\infty}^x\frac{{\rm d}}{{\rm d}y}(v^2(t,y)){\rm d}y-\int^{\infty}_x\frac{{\rm d}}{{\rm d}y}(v^2(t,y)){\rm d}y=\int_{-\infty}^x2v\cdot v_y(t,y){\rm d}y-\int^{\infty}_x2v\cdot v_y(t,y){\rm d}y\\
			&\leq \int_{-\infty}^x+\int^{\infty}_xv^2+v_y^2{\rm d}y\leq\int_{-\infty}^\infty v^2+v_y^2{\rm d}y=\|v(t,\cdot)\|_{H^1}^2\nonumber
		\end{aligned}
		$$
	\end{proof}
	
	\begin{lemm}
		Let $n_0\in H^s(\mathbb{R}),\ s>\frac 1 2$ and $n_0(x)\geq 0,\ \forall x\in\mathbb{R}$. Define an equivalent norm $H_{\beta_0}^1$ of $H^1$-norm:
		$$\|v\|_{H_{\beta_0}^1}\triangleq(\int_{\mathbb{R}}(4\beta_0^2+|\xi|^2)|\hat{v}(\xi)|^2{\rm d}\xi)^{1/2}$$
		Then the $H^1_{\beta_0}$-norm of $v$ in $[0,T_{n_0})$ is non-increasing, namely, if $0\leq t_1\leq t_2<T_{n_0}$
		$$\|v(t_2)\|_{H^1_{\beta_0}}\leq\|v(t_1)\|_{H^1_{\beta_0}}\leq\|v_0\|_{H^1_{\beta_0}}$$
	\end{lemm}
	\begin{proof}
		Since $n_0\geq0$, according to Lemma\ref{sign}, it follows that $v\geq 0$ and $|v_x(t,x)|\leq2|\beta_0|v(t,x)$. Hence for any $t\in[0,T_{n_0})$, we have
		$$\begin{aligned}
			\frac{{\rm d}}{{\rm d}t}\int_{\mathbb{R}}4\beta_0^2v^2+v_x^2{\rm d}x&=2\int_{\mathbb{R}}4\beta_0^2v\cdot v_t+v_x\cdot v_{xt}{\rm d}x=2\int_{\mathbb{R}}v(4\beta_0^2v_t-v_{xxt}){\rm d}x\\
			&=8\int_{\mathbb{R}}v[n(v_x+2\beta_0v)]_x{\rm d}x=-8\int_{\mathbb{R}}v_x(4\beta_0^2v-v_{xx})(v_x+2\beta_0v){\rm d}x\\
			&=-8\int_{\mathbb{R}}4\beta_0^2v\cdot v_x^2+8\beta_0^3v^2\cdot v_x-v_x^2v_{xx}-2\beta_0v\cdot v_x\cdot v_{xx}{\rm d}x\\
			&=-8\int_{\mathbb{R}}v_x^2(4\beta_0^2v+\beta_0v_x){\rm d}x=-8\int_{\mathbb{R}}v_x^2(|\beta_0|+\frac{\beta_0}2sgn(\cdot))e^{-|2\beta_0\cdot|\ast n{\rm d}x}\leq 0
		\end{aligned}$$
		Integrating the above inequality on $[0,t_1]$ and $[t_1,t_2]$ we have
		$$\|v(t_2)\|_{H^1_{\beta_0}}\leq\|v(t_1)\|_{H^1_{\beta_0}}\leq\|v_0\|_{H^1_{\beta_0}}$$
	\end{proof}
	
	\begin{theo}
		Let $n_0\in B^s_{2,2}$ with $s>1/2$ and $n_0\geq0$, the system of \eqref{Npeakon} has a unique global solution, i.e. $\forall T>0,\ n\in E^s_{2,2}(T)$.
	\end{theo}
	\begin{proof}
		Let $\psi(t,x)$ satisfies the flow equation \eqref{flow map}. Then by Lemma\ref{sign} and Sobolev embedding
		$$\frac{{\rm d}}{{\rm d}t}(n(t,\psi(t,x)))\leq(16\beta_0^2vn-4n^2+8\beta_0v_xn)(t,\psi(t,x))\leq 32\beta_0^2v\cdot n(t,\psi(t,x))\leq C\|v_0\|_{H^1}\cdot n(t,\psi(t,x))$$
		by Gronwall's inequality we have
		$$n(t,\psi(t,x))\leq n_0(x)\exp(C\|v_0\|_{H^1})$$
		thus $\forall T>0$, we have
		$$\int_0^T\|n(t,\cdot)\|_{L^\infty}{\rm d}t\leq \int_0^T\|n_0\|_{L^\infty}\exp(C\|v_0\|_{H^1}){\rm d}t$$
		
		Therefore by well-posedness and blow-up criterion of the system \eqref{Npeakon}, we prove that 
		$$n\in E^s_{2,2}(T),\ \ \forall T>0$$
	\end{proof}

	\section{Simulation}
	
	According to our blow up condition, we can construct an exact example of the system \eqref{Npeakon}.
	
	Let $\beta_0=1$ and denote
	\begin{equation}
		f(x)=
		\left\{
		\begin{array}{l}
			\exp(-\frac 1 {1-x^2})\quad |x|\leq1\\
			0\quad\quad\quad\quad\quad\ \ \ \   |x|>1\nonumber
		\end{array}
		\right.
	\end{equation}
	taking $n_0(x)=-20e\cdot f(20x)$, it is obviously that $\|n_0\|_{L^\infty}=20$. Since $\|n_0\|_{L^1}\leq \frac e 2$, by Young's inequality we can get the rough estimate of $v_0$ and $v_{0x}$
	$$\|v_0\|_{L^\infty}\leq \frac e 8,\ \|v_{0x}\|_{L^\infty}\leq \frac e 4$$
	then we can denote $T_1=\frac 1 {12e}$ and $b=8e^2$.
	
	It is obvious that
	$$-\sqrt{\frac b 2}+\frac{\sqrt{2b}}{1-\exp(2\sqrt{2b}T_1)}\geq -2e-\frac{4e}{1-\exp(3/4)}\geq -16\geq n_0(0)=-20$$
	then we have $n_0$ will blow up within $T_1\approx 0.0307$.

	The solution of system \eqref{Npeakon} with the initial data $n_0$ can be roughly expressed as Figure 1:
	\begin{figure}[H]\label{fig1}
		\centering 
		\includegraphics[height=6cm,width=9.5cm]{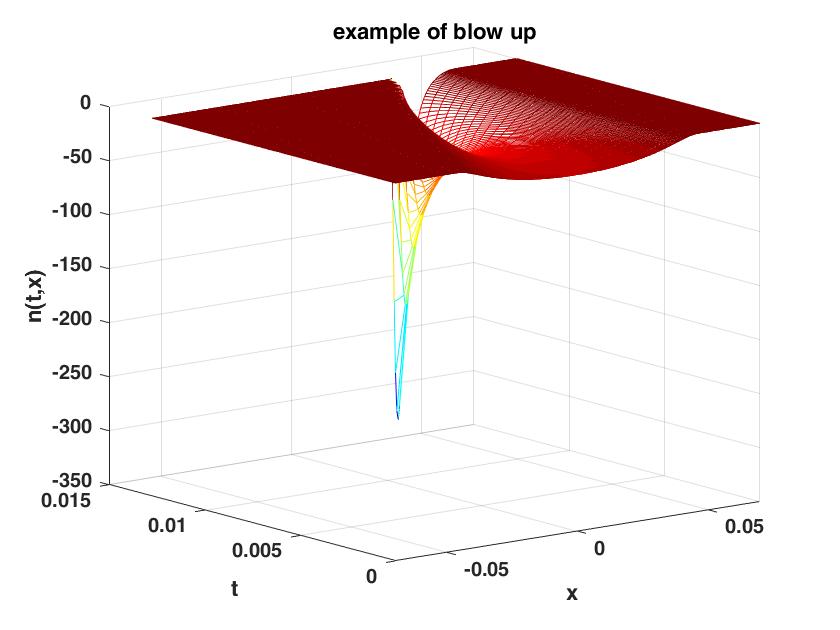}
		\caption{blow up}
	\end{figure}
	It is consistent with the approximate numerical calculation of blow up condition we have given.
	
	And for $n_0(x)=f(x)$, it is easy to check that $n_0\in B^s_{2,2}$ for all $s>1/2$ and $n_0(x)\geq0$, the solution of system \eqref{Npeakon} with the initial data $n_0$ can be roughly expressed as Figure 2:
	\begin{figure}[H]\label{fig2}
		\centering 
		\includegraphics[height=7cm,width=15cm]{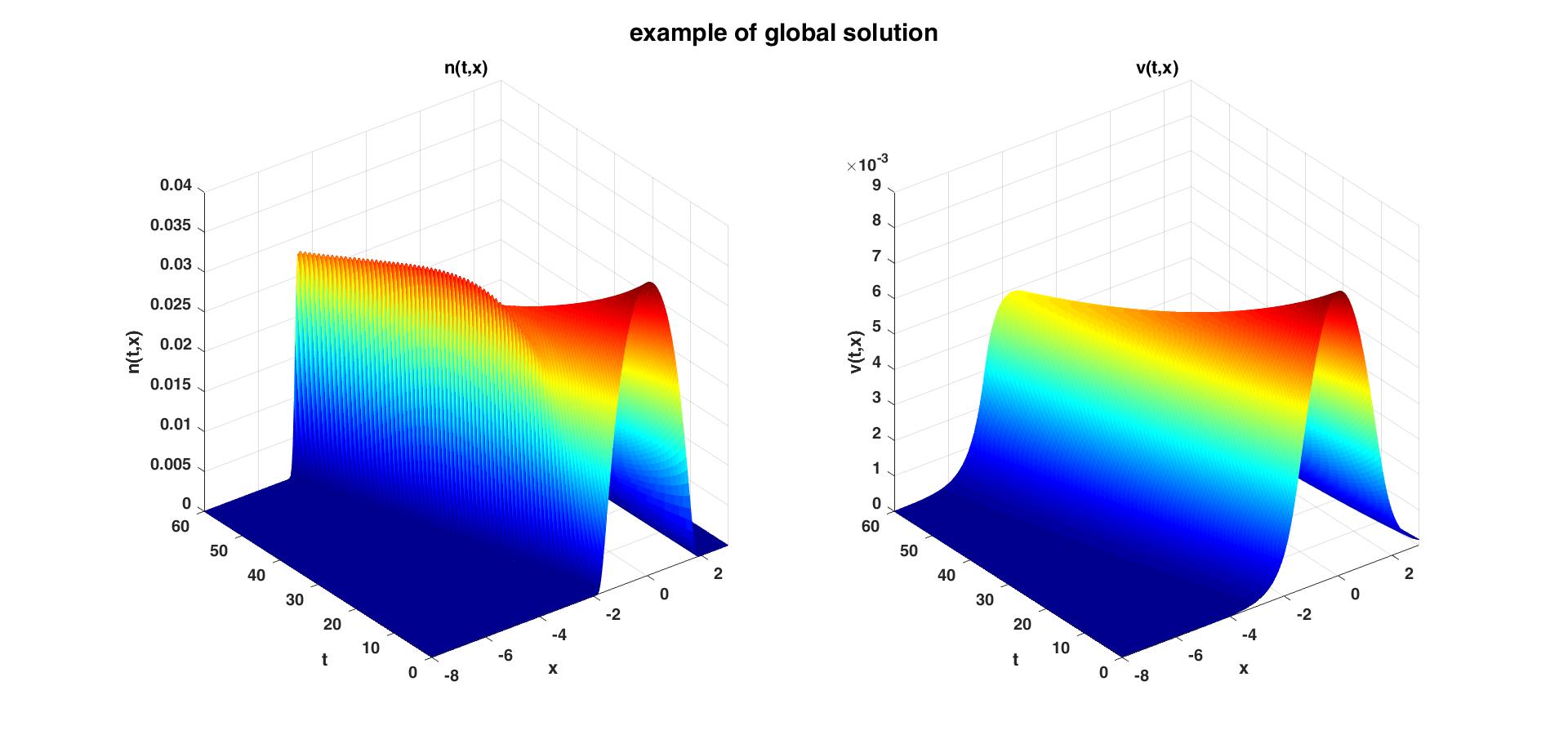}
		\caption{global solution}
	\end{figure}
	It is consistent with the existence of the global solution. And the above figure of the solution $v(t,x)$ is similar to the exact single peakon of the form\cite{Npeakon}
	$$v(t,x)=a_1e^{-2\beta_0|x+8\beta_0a_1t-a_2|}$$
	\begin{figure}[H]\label{fig3}
		\centering 
		\includegraphics[height=6cm,width=9.5cm]{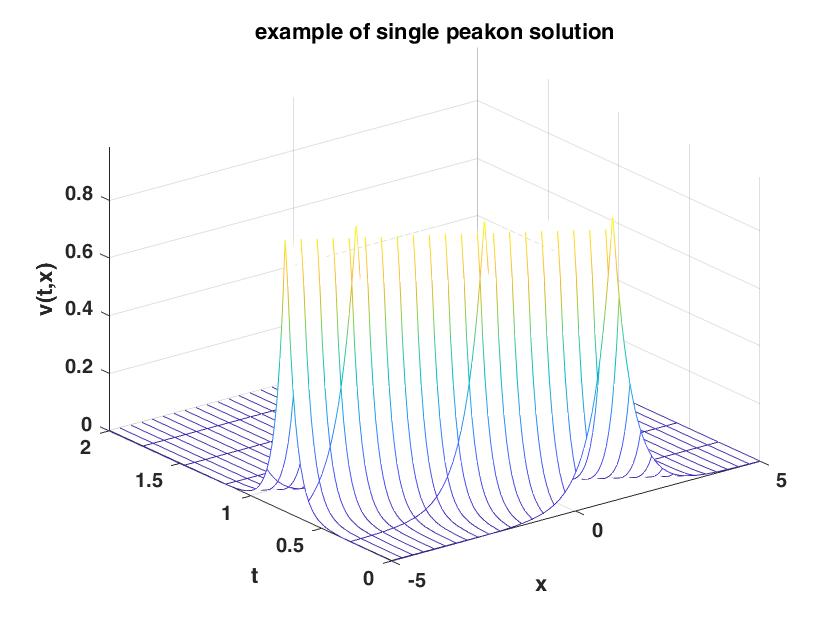}
		\caption{single peakon solution with $a_1=a_2=1,\beta_0=1$}
	\end{figure}
	
	\phantomsection
	\addcontentsline{toc}{section}{\refname}
	\bibliographystyle{abbrv} 
	\bibliography{Feneref}
\end{document}